\documentclass{amsart}
\usepackage{ifpdf}
\ifpdf\usepackage[colorlinks]{hyperref} 
\usepackage{microtype}
\else\usepackage[hypertex]{hyperref}\fi 
\usepackage{graphicx}
\usepackage{amscd}

\usepackage[curve]{xypic}
\newcommand{\walter}[1]{\marginpar{\sffamily{\noindent\tiny W: #1
\par}\normalfont}}
\renewcommand{\walter}[1]{}
\newbox\mybox
\def\overtag#1#2#3{\setbox\mybox\hbox{$#1$}\hbox to
  0pt{\vbox to 0pt{\vglue-#3\vglue-\ht\mybox\hbox to \wd\mybox
      {\hss$\ss#2$\hss}\vss}\hss}\box\mybox}
\def\undertag#1#2#3{\setbox\mybox\hbox{$#1$}\hbox to 0pt{\vbox to
    0pt{\vglue#3\vglue\ht\mybox\hbox to \wd\mybox
      {\hss$\ss#2$\hss}\vss}\hss}\box\mybox}
\def\lefttag#1#2#3{\hbox to 0pt{\vbox to 0pt{\vss\hbox to
      0pt{\hss$\ss#2$\hskip#3}\vss}}#1}
\def\righttag#1#2#3{\hbox to 0pt{\vbox to 0pt{\vss\hbox to
      0pt{\hskip#3$\ss#2$\hss}\vss}}#1}
\let\ss\scriptstyle

\def\Dot{\lower.2pc\hbox to 2.5pt{\hss$\bullet$\hss}}
\def\Circ{\lower.2pc\hbox to 2.5pt{\hss$\circ$\hss}}
\def\Vdots{\raise5pt\hbox{$\vdots$}}
\def\splicediag#1#2{\xymatrix@R=#1pt@C=#2pt@M=0pt@W=0pt@H=0pt}
\newcommand\lineto{\ar@{-}}
\newcommand\dashto{\ar@{--}}
\newcommand\dotto{\ar@{.}}

\newcommand{\C}{{\mathbb C}}
\newcommand{\Z}{{\mathbb Z}}
\newcommand{\R}{{\mathbb R}}

\newcommand{\Rplus}{{\mathbb R_+^*}}

\newtheorem{itheorem}{Theorem}
\newtheorem{icorollary}[itheorem]{Corollary}
\newtheorem{theorem}{Theorem}[section]
\newtheorem{lemma}[theorem]{Lemma}

\theoremstyle{definition}

\begin{document}
\title{Bi-Lipschitz geometry of weighted homogeneous surface singularities}
\author{Lev Birbrair}\thanks{Research supported under CNPq grant
no 300985/93-2}
\address{Departamento de Matem\'atica, Universidade Federal do Cear\'a
(UFC), Campus do Picici, Bloco 914, Cep. 60455-760. Fortaleza-Ce,
Brasil} \email{birb@ufc.br}
\author{Alexandre Fernandes}\thanks{Research supported under CNPq
grant no 300393/2005-9}
\address{Departamento de Matem\'atica, Universidade Federal do Cear\'a
(UFC), Campus do Picici, Bloco 914, Cep. 60455-760. Fortaleza-Ce,
Brasil} \email{alex@mat.ufc.br}
\author{Walter D.
  Neumann} \thanks{Research supported under NSA grant H98230-06-1-011
  and NSF grant no.\
  DMS-0206464} \address{Department of Mathematics, Barnard College,
  Columbia University, New York, NY 10027}
\email{neumann@math.columbia.edu}


\begin{abstract}
  We show that a weighted homogeneous complex surface singularity is
  metrically conical (i.e., bi-Lipschitz equivalent to a metric cone)
  only if its two lowest weights are equal. We also give an example of
  a pair of weighted homogeneous complex surface singularities that
  are topologically equivalent but not bi-Lipschitz equivalent.
\end{abstract}

\maketitle

\section{Introduction and main results}
A natural question of metric theory of singularities is the existence
of a metrically conical structure near a singular point of an
algebraic set. For example, complex algebraic curves, equipped with
the inner metric induced from an embedding in $\C^N$, always have
metrically conical singularities. It was discovered recently (see
\cite {BF}) that weighted homogeneous complex surface singularities
are not necessarily metrically conical. In this paper we 
show that they are rarely metrically conical.

Let $(V,p)$ be a normal complex surface singularity germ. Any set
$z_1,\dots,z_N$ of generators for $\mathcal O_{(V,p)}$ induces an
embedding of germs $(V,p)\to (\C^N,0)$. The Riemannian metric on
$V-\{p\}$ induced by the standard metric on $\C^N$ then gives a metric
space structure on the germ $(V,p)$.  This metric space structure, in
which distance is given by arclength within $V$, is called the
\emph{inner metric} (as opposed to \emph{outer metric} in which
distance between points of $V$ is distance in $\C^N$).

It is easy to see that, up to bi-Lipschitz equivalence, this inner
metric is independent of choices. It depends strongly on the
analytic structure, however, and may not be what one first expects.
For example, we shall see that if $(V,p)$ is a quotient singularity
$(V,p) = (\C^2/G,0)$, with $G\subset U(2)$ finite acting freely, then
this metric is usually not bi-Lipschitz equivalent to the conical
metric induced by the standard metric on $\C^2$.

If $M$ is a smooth compact manifold then a \emph{cone on $M$} will
mean the cone on $M$ with a standard Riemannian metric off the cone
point. This is the metric completion of the Riemannian manifold
$\R_+\times M$  with metric given (in terms of an element of arc
length) by $ds^2=dt^2+t^2ds_M^2$ where $t$ is the coordinate on
$\R_+$ and $ds_M$ is given by any Riemannian metric on $M$.
It is easy to see that this metric completion simply adds a single
point at $t=0$ and, up to bi-Lipschitz equivalence, the metric on the
cone is independent of choice of metric on $M$.  

If $M$ is the link of
an isolated complex singularity $(V,p)$ then the germ $(V,p)$ is
homeomorphic to the germ of the cone point in a cone $CM$. If this
homeomorphism can be chosen to be bi-Lipschitz we say, following
\cite{BGM}, that the germ $(V,p)$ is \emph{metrically conical}.
In \cite{BGM} the approach taken is to 
consider a semialgebraic triangulation of $V$ and consider the star of
$p$ according to this triangulation. The point $p$ is metrically
conical if the intersection $V\cap B_{\epsilon}[p]$ is bi-Lipschitz
homeomorphic to the star of $p$, considered with the standard metric
of the simplicial complex.

Suppose now that $(V,p)$ is weighted homogeneous. That is, $V$
admits a good $\C^*$--action (a holomorphic action with positive
weights: each orbit $\{\lambda x~|~\lambda\in \C^*\}$ approaches
zero as $\lambda\to 0$).  The weights $v_1, \dots, v_r$ of a minimal
set of homogeneous generators of the graded ring of $V$ are called
the \emph{weights of $V$}. We shall order them by size,
$v_1\ge\dots\ge v_r$, so $v_{r-1}$ and $v_r$ are the two lowest
weights. 

If $(V,p)$ is a cyclic quotient singularity $V=\C^2/\mu_n$ (where
$\mu_n$ denotes the $n$--th roots of unity) then it does not have a
unique $\C^*$--action. In this case we use the $\C^*$--action induced by the
diagonal action on $\C^2$.

If $(V,p)$ is homogeneous, that is, the weights $v_1,\dots, v_r$ are
all equal, then  it is easy to see that $(V,p)$ is metrically conical.

\begin{itheorem} \label{th:1} If the
  two lowest weights of $V$ are unequal then $(V,p)$ is not metrically conical.
\end{itheorem}

For example, the Kleinian singularities $A_k, k\ge 1$, $D_k, k\ge 4$,
$E_6$, $E_7$, $E_8$ are the quotient singularities $\C^2/G$ with
$G\subset SU(2)$ finite.  The diagonal action of $\C^*$ on $\C^2$
induces an action on $\C^2/G$, so they are weighted homogeneous. They
are the weighted homogeneous hypersurface singularities:

\smallskip\begin{tabular}{lrl}
&\bf equation\hbox to12pt{\hss}&\bf \hbox to12pt{\hss}weights\\
$A_k:$&$\quad x^2+y^2+z^{k+1}=0$&$(k+1,k+1,2)$ or
$(\frac{k+1}2,\frac{k+1}2,1)$\\
$D_k$:&$\quad x^2+y^2z+z^{k-1}=0$&$(k-1,k-2,2)$, $k\ge 4$\\
$E_6$:&$x^2+y^3+z^4=0$&$(6,4,3)$\\
$E_7$:&$x^2+y^3+yz^3=0$&$(9,6,4)$\\
$E_8$:&$x^2+y^3+z^5=0$&$(15,10,6)$    
\end{tabular}

\smallskip\noindent By the theorem, none of them is metrically conical
except for the quadric $A_1$ and possibly\footnote{We have a
  tentative proof that the quaternion quotient is metrically conical,
  see \cite{inprep}.} the quaternion group quotient $D_4$.


The general cyclic quotient singularity is of the form $V=\C^2/\mu_n$
where the $n$--th roots of unity act on $\C^2$ by
$\xi(u_1,u_2)=(\xi^q u_1,\xi u_2)$ for some $q$ prime to $n$ with
$0<q<n$; the link of this singularity is the lens space $L(n,q)$. It
is homogeneous if and only if $q=1$.

\begin{itheorem}\label{th:1.5}
  A cyclic quotient singularity is metrically conical if and only if
  it is homogeneous.
\end{itheorem}

Many non-homogeneous cyclic quotient singularities have their two lowest
weights equal, so the converse to Theorem \ref{th:1}
is not generally true.

We can also sometimes distinguish weighted homogeneous singularities
with the same topology from each other.

\begin{itheorem} \label{th:2} Let $(V,p)$ and $(W,q)$ be two weighted
  homogeneous normal surface singularities, with weights $v_1\ge
  v_2\ge \dots\ge v_r$ and $w_1\ge w_2\ge \dots\ge w_s$ respectively.
 If either $\frac{v_{r-1}}{v_r}>\frac{w_1}{w_s}$ or
  $\frac{w_{s-1}}{w_s}>\frac{v_1}{v_r}$ then
$(V,p)$ and $(W,q)$ are not bi-Lipschitz homeomorphic.
  \end{itheorem}

\begin{icorollary} Let $V,W\subset\C^3$ be defined by
  $V=\{(z_1,z_2,z_3)\in\C^3~:~z_1^{2}+z_2^{51}+z_3^{102}=0\}$ and
  $W=\{(z_1,z_2,z_3)\in\C^3~:~z_1^{12}+z_2^{15}+z_3^{20}=0\}$.
  Then, the germs $(V,0)$ and $(W,0)$ are homeomorphic, but they are
  not bi-Lipschitz homeomorphic.
\end{icorollary}

The corollary follows because in both cases the link of the
singularity is an $S^1$ bundle of Euler class $-1$ over a curve of
genus 26;  the weights are $(51,2,1)$ and $(5,4,3)$ respectively and
Theorem 2 applies since $\frac21>\frac53$.

\smallskip The idea of the proof of Theorem \ref{th:1} is to find an
essential closed curve in $V-\{p\}$ with the property that as we
shrink it towards $p$ using the $\R^*$ action, its diameter shrinks
faster than it could if $V$ were bi-Lipschitz equivalent to a
cone. Any essential closed curve in $V-\{p\}$ that lies in the
hyperplane section $z_r=1$ will have this property, so we must show
that the hyperplane section contains such curves. The proofs of
Theorems \ref{th:1.5} and \ref{th:2} are similar.

\section{Proofs}

Let $z_1,\dots,z_r$ be a minimal set of homogeneous generators of
the graded ring of $V$, with $z_i$ of weight $v_i$ and $v_1\ge
v_2\ge\dots v_{r-1}\ge v_r$. Then $x\mapsto (z_1(x),\dots,z_r(x))$
embeds $V$ in $\C^r$. This is a $\C^*$--equivariant embedding for
the $\C^*$--action on $\C^r$ given by
$z(z_1,\dots,z_r)=(z^{v_1}z_1,\dots,z^{v_r}z_r)$

Consider the subset $V_0:=\{x\in V~|~z_r(x)=1\}$ of $V$. This is a
nonsingular complex curve.

\begin{lemma}\label{le:1} Suppose 
  $(V,p)$ is not a homogeneous cyclic quotient singularity. Then for any
  component $V'_0$ of $V_0$ the map $\pi_1(V'_0)\to \pi_1(V-\{p\})$ is
  non-trivial.
\end{lemma}


\begin{proof}
  Denote $v=lcm(v_1,\dots,v_r)$.  A convenient version of the link of
  the singularity is given by
  \[M=S\cap V \quad\text{with}\quad S=\{z\in
  \C^r~|~|z_1|^{2v/v_1}+\dots+|z_r|^{2v/v_r}=1\}\,.\] The action of
  $S^1\subset \C^*$ restricts to a fixed-point free action on $M$. If
  we denote the quotient $M/S^1=(V-\{p\})/\C^*$ by $P$ then the orbit
  map $M\to P$ is a Seifert fibration, so $P$ has the structure of an
  orbifold. The orbit map induces a surjection of
  $\pi_1(V-\{p\})=\pi_1(M)$ to the orbifold fundamental group
  $\pi_1^{orb}(P)$ (see eg \cite{NR, Scott}) so the lemma will follow
  if we show the image of $\pi_1(V'_0)$ in $\pi_1^{orb}(P)$ is
  nontrivial.

  Denote $V_r:=\{z\in V~|~z_r\ne 0\}$ and $P_r:=\{[z]\in P~|~z_r\ne
  0\}$ and $\pi\colon V\to P$ the projection. Each generic orbit of
  the $\C^*$--action on $V_r$ meets $V_0$ in $v_r$ points; in fact
  the $\C^*$--action on $V_r$ restricts to an action of $\mu_{v_r}$
  (the $v_r$--th roots of unity) on $V_0$, and
  $V_0/\mu_{v_r}=V_r/\C^*=P_r$.  Thus $V_0\to P_r$ is a cyclic cover
  of orbifolds, so the same is true for any component $V'_0$ of $V_0$.
  Thus $\pi_1(V'_0)\to \pi_1^{orb}(P_r)$ maps $\pi_1(V'_0)$
  injectively to a normal subgroup with cyclic quotient.
  On the other hand $\pi_1^{orb}(P_r)\to \pi_1^{orb}(P)$ is
  surjective, since $P_r$ is the complement of a finite set of points
  in $P$. Hence, the image of $\pi_1(V'_0)$ in $\pi_1^{orb}(P)$ is a
  normal subgroup with cyclic quotient. Thus the lemma follows if
  $\pi_1^{orb}(P)$ is not cyclic.

  If $\pi_1^{orb}(P)$ is cyclic
  then $P$ is a $2$--sphere with at most two orbifold points, so the
  link $M$ must be a lens space, so $(V,p)$ is a cyclic quotient
  singularity, say $V=\C^2/\mu_n$. Here $\mu_n$ acts on $\C^2$ by
  $\xi(u_1,u_2)=(\xi^q u_1, \xi u_2)$ with $\xi=e^{2\pi i/n}$, for some
  $0<q<n$ with $q$ prime to $n$.

  Recall that we are using the diagonal $\C^*$--action. The base
  orbifold is then $(\C^2/\mu_n)/\C^*=(\C^2/\C^*)/\mu_n=
  P^1\C/\mu_n$. Note that $\mu_n$ may not act effectively on $P^1\C$;
  the kernel of the action is
  \begin{align*}
    \mu_n\cap\C^*&=\{(\xi^{qa},\xi^a~)|~ \xi^{qa}=\xi^a\}  \\
    &=\{(\xi^{qa},\xi^a~)|~ \xi^{(q-1)a}=1\}  \\
    &= \mu_{d}\quad \text{with }d=\gcd(q-1,n)\,.
  \end{align*}
  So the actual action is by a cyclic group of order $n':=n/d$ and the
  orbifold $P$ is $P^1\C/(\Z/n')$, which is a $2$--sphere with two
  degree $n'$ cone points.

  The ring of functions on $V$ is the ring of invariants for the
  action of $\mu_n$ on $\C^2$, which is generated by functions of the
  form $u_1^au_2^b$ with $qa+b\equiv 0$ (mod $n$). The minimal set of
  generators is included in the set consisting of $u_1^n$, $u_2^n$,
  and all $u_1^au_2^b$ with $qa+b\equiv 0$ (mod $n$) and $0<a,b<n$. If
  $q=1$ these are the elements $u_1^au_2^{n-a}$ which all have equal
  weight, and $(V,p)$ is homogeneous and a cone; this case is excluded by
  our assumptions. If $q\ne 1$ then a generator of least weight will
  be some $u_1^au_2^b$ with $a+b<n$. Then $V_0$ is the subset of $V$
  given by the quotient of the set $\bar V_0=\{(u_1,u_2)\in
  \C^2~|~u_1^au_2^b=1\}$ by the $\mu_n$--action.  Each fiber of the
  $\C^*$--action on $\C^2$ intersects $\bar V_0$ in exactly $a+b$
  points, so the composition $\bar V_0\to \C^2-\{0\}\to P^1\C$ induces
  an $(a+b)$--fold
  covering $\bar V_0\to P^1\C-\{0,\infty\}$. Note that $d=\gcd(q-1,n)$
  divides $a+b$ since $a+b=(qa+b)-(q-1)a=nc-(q-1)a$ for some $c$.
  Hence the subgroup $\mu_d=\mu_n\cap\C^*$ is in the covering
  transformation group of the above covering, so the covering $V_0\to
  P_0$ obtained by quotienting by the $\mu_n$--action has degree at
  most $(a+b)/d$. Restricting to a component $V'_0$ of $V_0$ gives us
  possibly smaller degree. Since $(a+b)/d<n/d=n'$, the image of
  $\pi_1(V'_0)$ in $\pi_1(P)=\Z/n'$ is non-trivial, completing the
  proof.
\end{proof}

\begin{proof}[Proof of Theorem \ref{th:1}]
  Assume $v_{r-1}/v_r>1$.  By Lemma \ref{le:1} we can find a closed
  curve $\gamma$ in $V_0$ which represents a non-trivial element of
  $\pi_1(V-\{p\})$. Suppose we have a bi-Lipschitz homeomorphism $h$
  from a neighborhood of $p$ in $V$ to a neighborhood in the cone
  $CM$. Using the $\Rplus$--action on $V$, choose $\epsilon>0$ small
  enough that  $t\gamma$
  is in the given neighborhood of $p$ for $0<t\le \epsilon$.

  Consider the map $H$ of $[0,1]\times (0,\epsilon]$ to $V$ given by
  $H(s,t)=t^{-v_r}h(t\gamma(s))$.  Here $t\gamma(s)$ refers to the
  $\Rplus$--action on $V$, and $t^{-v_r}h(v)$ refers to the
  $\Rplus$--action on $CM$.  Note that the coordinate $z_r$ is
  constant equal to $ t^{v_r}$ on each $t\gamma$ and the other
  coordinates have been multiplied by at most $t^{v_{r-1}}$.  Hence,
  for each $t$ the curve $t\gamma$ is a closed curve of length of
  order bounded by $t^{v_{r-1}}$, so $h(t\gamma)$ has length of the
  same order, so $t^{-v_r}h(t\gamma)$ has length of order
  $t^{v_{r-1}-v_r}$.  This length approaches zero as $t\to 0$, so $H$
  extends to a continuous map $H'\colon[0,1]\times [0,\epsilon]\to V$
  for which $H([0,1]\times\{0\})$ is a point.  Note that
  $t^{-v_r}h(t\gamma)$ is never closer to $p$ than distance $1/K$,
  where $K$ is the bi-Lipschitz constant of $h$, so the same is true
  for the image of $H'$. Thus $H'$ is a null-homotopy of
  $\epsilon\gamma$ in $V-\{p\}$, contradicting the fact that $\gamma$
  was homotopically nontrivial.
\end{proof}
\begin{proof}[Proof of Theorem \ref{th:1.5}]
  Suppose $(V,p)$ is a non-homogeneous cyclic quotient singularity, as
  in the proof of Lemma \ref{le:1} and suppose Theorem \ref{th:1} does
  not apply, so the two lowest weights are equal (in the notation of
  that proof this happens, for example, if $n=4k$ and $q=2k+1$ for
  some $k> 1$: the generators of the ring of functions of lowest
  weight are $u_1u_2^{2k-1}, u_1^3u_2^{2k-3}, \dots$, of weight $2k$).
  Let $u_1^au_2^b$ be the generator of lowest weight that has smallest
  $u_1$--exponent and choose this one to be the coordinate $z_r$ in
  the notation of Lemma \ref{le:1}. Consider now the $C^*$--action
  induced by the action $t(u_1,u_2) = (t^\alpha u_1,t^\beta u_2)$ on
  $\C^2$ for some pairwise prime pair of positive integers
  $\alpha>\beta$. With respect to this $\C^*$--action the weight
  $\alpha a'+\beta b'$ of any generator $u_1^{a'}u_2^{b'}$ with $a'>a$
  will be greater that the weight $\alpha a+\beta b$ of $z_r$ (since
  $a'+b'\ge a+b$, which implies $\alpha a'+\beta b'=\alpha a +
  \alpha(a'-a)+\beta b'>\alpha a + \beta(a'-a)+\beta b'\ge \alpha
  a+\beta b$). On the other hand, any generator $u_1^{a'}u_2^{b'}$
  with $a'<a$ had $a'+b'>a+b$ by our choice of $z_r$, and if
  $\alpha/\beta$ is chosen close enough to $1$ we will still have
  $\alpha a'+\beta b'>\alpha a+\beta b$, so it will still have larger
  weight than $z_r$. Thus $z_r$ is then the unique generator of lowest
  weight, so we can carry out the proof of Theorem \ref{th:1} using
  this $\C^*$--action to prove non-conicalness of the singularity.
\end{proof}
\begin{proof}[Proof Theorem \ref{th:2}] Let $h\colon (V,p)\rightarrow
  (W,q)$ be a $K$--bi-Lipschitz homeomorphism.  Let us suppose that
  $\frac{v_{r-1}}{v_r}>\frac{w_1}{w_s}$.  Let $\gamma$ be a loop in
  $V_0$ representing a non-trivial element of $\pi_1(V-\{p\})$ (see
  Lemma \ref{le:1}). We choose $\epsilon$ as in the previous proof.
  For $t\in (0,\epsilon]$ consider the curve $t\gamma$, where
  $t\gamma$ refers to $\Rplus$-action on $V$.  Its length
  $l(t\gamma)$, considered as a function of $t$, has the order bounded
  by $t^{v_{r-1}}$. The distance of the curve $t\gamma$ from $p$ is of
  order $t^{v_r}$.  Since $h$ is a bi-Lipschitz map, we obtain the
  same estimates for $h(t\gamma)$.  Since the smallest weight for $W$
  is $w_s$, the curve $t^{-v_r/w_s}h(t\gamma)$ will be distance at
  least $1/K$ from $p$.  Moreover its length will be of order at most
  $t^{-w_1v_r/w_s}l(t\gamma)$ which is of order
  $t^{v_{r-1}-w_1v_r/w_s}$. This approaches zero as $t\to 0$ so, as in
  the previous proof, we get a contradiction to the non-triviality of
  $[\gamma]\in\pi_1(V-\{p\})=\pi_1(W-\{q\})$. By exchanging the roles
  of $V$ and $W$ we see that $\frac{w_{s-1}}{w_s}>\frac{v_1}{v_r}$
  also leads to a contradiction.
\end{proof}

\end{document}